\documentclass[11pt,twoside,reqno]{amsart}

\usepackage{mathtools}
\usepackage{todonotes}
\usepackage{microtype}
\usepackage{cite}
\usepackage[OT1]{fontenc}
\usepackage{type1cm}
\usepackage{amssymb}
\usepackage{dsfont}
\usepackage{enumitem}
\usepackage{comment}
\usepackage{xcolor}
\usepackage{lmodern}

\usepackage{geometry}
\usepackage{hyperref}
\hypersetup{
  colorlinks=true,
  linkcolor=black,
  anchorcolor=black,
  citecolor=black,
  filecolor=black,      
  menucolor=red,
  runcolor=black,
  urlcolor=black,
}
\usepackage[capitalize,noabbrev,nameinlink]{cleveref}

\numberwithin{equation}{section}
\crefformat{equation}{(#2#1#3)}
\crefformat{figure}{#2Figure #1#3}
\crefformat{enumi}{(#2#1#3)}
\crefformat{theorem}{#2Theorem~#1#3}
\crefformat{lemma}{#2Lemma~#1#3}
\crefformat{proposition}{#2Proposition~#1#3}
\crefformat{corollary}{#2Corollary~#1#3}
\crefformat{claim}{#2Claim~#1#3}
\crefformat{section}{#2§#1#3}
\crefformat{subsection}{#2§#1#3}

\theoremstyle{plain}
\newtheorem{theorem}{Theorem}[section]

\newtheorem{proposition}[theorem]{Proposition}

\newtheorem{lemma}[theorem]{Lemma}
\newtheorem{conjecture}[theorem]{Conjecture}

\theoremstyle{remark}

\theoremstyle{definition}
\newtheorem{definition}[theorem]{Definition}
\newtheorem*{definition*}{Definition}
\newtheorem{question}[theorem]{Question}

\newcommand{\ttt}{\mathbf{t}}

\newcommand{\LL}{\mathcal{L}}

\renewcommand{\SS}{\mathcal{S}}

\newcommand{\CC}{\mathcal{C}}

\newcommand{\WW}{\mathcal{W}}

\newcommand{\R}{\mathbb{R}}

\newcommand{\Z}{\mathbb{Z}}
\newcommand{\N}{\mathbb{N}}

\newcommand{\roo}{\varrho}

\newcommand{\dd}{\,\mathrm{d}}

\renewcommand{\ge}{\geqslant}
\renewcommand{\le}{\leqslant}
\renewcommand{\geq}{\geqslant}
\renewcommand{\leq}{\leqslant}

\DeclareMathOperator{\ldimloc}{\underline{dim}_{loc}}

\DeclareMathOperator{\udimm}{\overline{dim}_M}

\DeclareMathOperator{\dimh}{dim_H}

\DeclareMathOperator{\dist}{dist}

\DeclareMathOperator{\proj}{proj}

\DeclareMathOperator{\spt}{spt}

\DeclareMathOperator{\id}{Id}
\DeclareMathOperator{\graph}{Gr}

\begin{document}

\title{Level sets of prevalent H\"older functions}

\author{Roope Anttila}
\address[Roope Anttila]
        {Research Unit of Mathematical Sciences \\ 
         P.O.\ Box 8000 \\ 
         FI-90014 University of Oulu \\ 
         Finland}
\email{roope.anttila@oulu.fi}

\author{Bal\'azs B\'ar\'any}
\address[Bal\'azs B\'ar\'any]
        {Department of Stochastics \\
        	Institute of Mathematics \\
        	Budapest University of Technology and Economics \\
        	M\H{u}egyetem rkp. 3 \\
        	H-1111 Budapest,
        	Hungary}
\email{balubsheep@gmail.com}

\author{Antti K\"aenm\"aki}
\address[Antti K\"aenm\"aki]
        {Alfr\'ed R\'enyi Institute of Mathematics \\
         Hungarian Academy of Sciences \\ 
         Budapest \\ 
         Hungary}
\email{antti.kaenmaki@gmail.com}

\date{\today}
\thanks{R. Anttila was supported by the Magnus Ehrnrooth foundation. B. B\'ar\'any was supported by the grants NKFI FK134251, K142169, and the grant NKFI KKP144059 ``Fractal geometry and applications''.}

\begin{abstract}
  We study the level sets of prevalent Hölder functions. For a prevalent $\alpha$-H\"older function on the unit interval, we show that the upper Minkowski dimension of every level set is bounded from above by $1-\alpha$ and Lebesgue positively many level sets have Hausdorff dimension equal to $1-\alpha$.
\end{abstract}

\maketitle

\section{Introduction}\label{sec:intro}
Intuitively, regularity of functions is related to size and complexity of their graphs---the more regular a function is, the smaller one expects the dimension of its graph to be. Perhaps the most popular notion of uniform regularity for real functions is Hölder regularity. A function $f\colon [0,1]\to \R$ is \emph{$\alpha$-Hölder} if there is a constant $C>0$ such that for all $x,y\in [0,1]$, we have
\begin{equation*}
| f(x)-f(y)|\leq C|x-y|^\alpha.
\end{equation*} 
Indeed, $\alpha$-Hölder regularity gives an upper bound of $2-\alpha$ for the upper Minkowski dimension of the graph; see \cite{MR1102677}. For many classically studied Hölder functions in fractal geometry, such as the Weierstrass and Takagi functions, the Hölder condition is in some sense sharp and for these functions, the Minkowski dimension exists and is equal to $2-\alpha$; see e.g. \cite[§11]{MR1102677}. While the Minkowski dimension of the graph is relatively easy to calculate with fairly elementary methods, calculating the Hausdorff dimension of the graph usually requires more sophisticated tools. The long standing open problem of determining the Hausdorff dimension of the graph was recently solved by Bárány, Hochman, and Rapaport \cite{BaranyHochmanRapaport2019} for the Takagi function and by Shen \cite{Shen2018} for the Weierstass function; see also Shen and Ren \cite{RenShen2021} who proved a nice dichotomy result for more general Weierstrass-type functions. The proofs of these results depend heavily on the dynamics underlying the graphs of these functions and are based on careful transversality estimates in \cite{Shen2018} and powerful entropy methods introduced by Hochman \cite{Hochman2014} in \cite{BaranyHochmanRapaport2019,RenShen2021}.

A common theme in many areas of mathematics is that while it is often difficult to obtain information on specific objects, allowing slight perturbations in their structure considerably simplifies the situation. This will be the case in this paper as well: we will study the Hausdorff dimension of the graphs and level sets of typical Hölder functions. In functional analysis, there are two common definitions for typicality: a topological notion of Baire typicality and a measure theoretic notion of prevalence.
\begin{definition}
  A subset $A$ of a Banach space $X$ is \emph{meagre} if it is a countable union of nowhere dense sets. A set $A\subset X$ is \emph{comeagre} or \emph{Baire typical} if its complement is meagre.
\end{definition}

\begin{definition}
  A subset $A$ of a Banach space $X$ is \emph{shy} if there exists a Borel measure $\mu$ and a Borel set $B\subset X$, such that $A\subset B$, $0<\mu(K)<\infty$ for some compact set $K\subset X$ and $\mu(x+B)=0$, for all $x\in X$. A set $A$ is \emph{prevalent} if its complement is shy.
\end{definition}
The definition of prevalence is rather abstract and a more concrete way of proving that a set $A\subset X$ is prevalent, which we will employ in this paper, is by showing the existence of a $k$-dimensional subspace $S \subset X$, called the \emph{probe space}, with the property that for any $x \in X$, we have $x+s \in B$ for $\LL$-almost all $s \in S$, where $\LL$ is the Lebesgue measure on $S$ and $B$ is a Borel set contained in $A$. In this case, we say that $A$ is \emph{$k$-prevalent}. It is straightforward from the definitions that a set admitting a probe space is prevalent: Indeed, for $\mu=\LL$ one can take $K$ to be the closed unit ball in $S$ and then
\begin{equation*}
  \mu(x+X\setminus B)=\LL(\{s\in S \colon -x+s\not\in B\})=0
\end{equation*}
for all $x \in X$.

Curiously, the geometric properties of typical continuous functions depend heavily on which notion of typicality is used. Going forward, for any real function $f$ we let $\graph(f)$ denote the graph of $f$ and we denote the Banach space of continuous real functions on the unit interval by $\CC([0,1])$. The first statement in the following theorem was proved by Mauldin and Williams \cite{MauldinWilliams1986} and the second one by Fraser and Hyde \cite{FraserHyde2012}.
\begin{theorem}[Mauldin-Williams, Fraser-Hyde]\leavevmode\label{thm:prevalent-continuous}
  \begin{enumerate}
    \item\label{it:prevalent-continuous-1} A Baire typical function $f\in \CC([0,1])$ satisfies
    \begin{equation*}
      \dimh(\graph(f))=1.
    \end{equation*}
    \item\label{it:prevalent-continuous-2} A prevalent function $f\in \CC([0,1])$ satisfies
    \begin{equation*}
      \dimh(\graph(f))=2.
    \end{equation*}
  \end{enumerate}
\end{theorem}
These results highlight a striking difference in the behaviour of typical functions when typicality is interpreted in the topological sense or in the measure theoretic sense. Topologically, typical continuous functions seem to be fairly regular since their graphs have the smallest possible Hausdorff dimension while in the measure theoretic sense, typical functions are highly irregular as their graphs have maximal Hausdorff dimension. Since one intuitively expects typical functions to behave in quite a chaotic fashion, prevalence is perhaps the better notion of typicality when studying the dimensions of the graphs of functions.

For Hölder functions, the following result, which is analogous to \cref{thm:prevalent-continuous}\cref{it:prevalent-continuous-2}, was originally proved by Clausel and Nicolay \cite{ClauselNicolay2010}, but we provide a short proof in \cref{prop:graph-big}; see also \cite{BayartHeurteaux2013} for a generalisation. Recall that the family of all real valued $\alpha$-H\"older functions defined on $[0,1]$, which we denote by $C^\alpha([0,1])$, equipped with the standard norm 
\begin{equation*}
  \|f\|_\alpha = \sup_{x \in [0,1]}|f(x)| + \sup_{(x,y) \in [0,1]^2} \frac{|f(x)-f(y)|}{|x-y|^\alpha}
\end{equation*}
is a Banach space.

\begin{theorem}[Clausel-Nicolay]\label{thm:prevalent-dim}
  A prevalent function $f\in \CC^{\alpha}([0,1])$ satisfies
  \begin{equation*}
    \dimh(\graph(f))=2-\alpha.
  \end{equation*}
\end{theorem}

A related approach for studying the complexity of a function $f$ is looking at the structure of its level sets $f^{-1}(\{y\})$. It is expected that, if there is no obvious obstruction, the dimension of the level set $f^{-1}(\{y\})$ should be at most $\dimh(\graph(f))-1$. Indeed, it follows from the classical Marstrand's slicing theorem, see i.e. \cite[Theorem 1.6.1]{BishopPeres2016}, that
\begin{equation}\label{eq:slicing}
  \dimh(f^{-1}(\{y\}))\leq \dimh(\graph(f))-1
\end{equation}
for Lebesgue almost every $y\in f([0,1])$. For the Takagi function, the authors proved in \cite{AnttilaBaranyKaenmaki2023}, that if one replaces the Hausdorff dimension by the Assouad dimension (see \cite{Fraser2021} for the definition) on the right hand side of \cref{eq:slicing}, then the upper bound holds for \emph{all} slices of the graph and is sharp.

For prevalent continuous functions, the following result on the size of the level sets was proved by Balka, Darji, and Elekes \cite{BalkaDarjiElekes2016}. The Lebesgue measure on $\R^d$ is denoted by $\LL^d$.
\begin{theorem}[Balka-Darji-Elekes]\label{thm:balka-darji-elekes}
  For a prevalent $f\in\CC([0,1])$, there is an open set $U_f\subset f([0,1])$ with $\LL^1(f^{-1}(U_f))=1$ such that for all $y\in U_f$ we have
  \begin{equation*}
    \dimh(f^{-1}(\{y\}))=1.
  \end{equation*}
\end{theorem}

Questions were raised in \cite[Problem 9.2]{BalkaDarjiElekes2016}, whether similar results hold for prevalent Hölder functions. In the main result of this paper, we answer the second question of \cite[Problem 9.2]{BalkaDarjiElekes2016} by showing that for a prevalent $\alpha$-Hölder function, the Hausdorff dimension of Lebesgue positively many level sets is $1-\alpha$. Moreover, we strengthen \cref{eq:slicing} by showing that for prevalent Hölder functions the upper Minkowski dimension of \emph{every} level set is at most $1-\alpha$.

\begin{theorem}[Main Theorem] \label{thm:main}
  A prevalent function $f\in C^{\alpha}([0,1])$ satisfies
  \begin{enumerate}
    \item\label{it:main1} $\udimm(f^{-1}(\{y\})) \le 1-\alpha$ for all $y \in \R$ provided that $0<\alpha<\frac12$,
    \item\label{it:main2} $\LL^1(\{y \in f([0,1]) \colon \dimh(f^{-1}(\{y\})) = 1-\alpha\})>0$ provided that $0<\alpha<1$.
  \end{enumerate}
\end{theorem}

The requirement $\alpha<\frac12$ in the first claim is a technical assumption which we need to prove the $k$-prevalence and, in particular, to show that the average Fourier decay of a certain measure is finite in the proof of \cref{thm:main-tech}. To overcome this obstacle and extend the result to $\frac12 \le \alpha < 1$, it might be possible to adapt methods on fractional Brownian motion by Erraoui and Hakiki \cite{ErraouiHakiki2022}. The proof of \cref{thm:main}, to which \cref{sec:proof} is devoted, relies on a combination of techniques from potential theory and Fourier analysis of measures.

\subsection{Weierstrass-type functions}
The original motivation for this paper was to study prevalent functions from the family of Weierstrass-type functions defined as follows. For $b\in\N$, $0<\alpha<1$, and a Lipschitz function $g$ on $\mathbb{S}^1$, define the \emph{Weierstrass-type function} $W_g^{\alpha,b}\colon \mathbb{S}^1\to\R$ by
\begin{equation*}
  W_g^{\alpha,b}(x)=\sum_{k=0}^{\infty}b^{-\alpha k}g(b^kx).
\end{equation*}
Given $b\in\N$ and $0<\alpha<1$ the family of corresponding Weierstass-type functions is denoted by $\WW^{\alpha,b}$. Note that the functions in this space are in one to one correspondence with the space $\mathrm{Lip}(\mathbb{S}^1)$ of Lipschitz functions on $\mathbb{S}^1$, which form a Banach space under the norm
\begin{equation*}
  \|f\|_{\mathrm{Lip}} = \sup_{x \in \mathbb{S}^1}|f(x)| + \sup_{(x,y) \in \mathbb{S}^1\times \mathbb{S}^1} \frac{|f(x)-f(y)|}{|x-y|}.
\end{equation*}
Thus we may identify $\WW^{\alpha,b}$ with the space $\mathrm{Lip}(\mathbb{S}^1)$ and when we talk about a prevalent function in $\WW^{\alpha,b}$ we mean prevalence under this identification. The family $\WW^{\alpha,b}$ includes many well studied functions such as the original Weierstass function where $g(x)=\cos(2\pi x)$ or the Takagi function where $g(x)=\dist(x,\Z)$ which were mentioned earlier in the introduction. We make the following conjecture:
\begin{conjecture} \label{conj:main}
  For any $b\in\N$ and $0<\alpha<\frac12$, a prevalent function $g \in \mathrm{Lip}(\mathbb{S}^1)$ satisfies
  \begin{enumerate}
    \item $\udimm((W_g^{\alpha,b})^{-1}(\{y\})) \le 1-\alpha$ for all $y \in \R$,
    \item $\dimh((W_g^{\alpha,b})^{-1}(\{y\})) = 1-\alpha$ for $\LL^1$-almost all $y\in W_g^{\alpha,b}([0,1])$.
  \end{enumerate}
\end{conjecture}
The methods of this paper provide an outline for the proof but unfortunately our efforts in proving a key technical result fall short. We discuss how to adapt the methods of this paper for the setting of Weierstrass-type functions and where the argument fails in \cref{sec:discussion}. We also formulate an explicit open question, \cref{ques:bi-holder}, a positive answer to which implies \cref{conj:main}.

\section{Preliminaries}
In this paper, a measure refers to a finite Borel measure with compact support on $\R^d$, unless stated otherwise. The pushforward of $\mu$ is denoted by $f_*\mu$ whenever $f$ is measurable.

For a bounded set $A \subset \R^d$, we denote the Hausdorff dimension of $A$ by $\dimh(A)$ and the upper Minkowski dimension of $A$ by $\udimm(A)$. For the upper Minkowski dimension, we use the definition via packings, that is,
\begin{equation*}
  \udimm(A)=\limsup_{r\downarrow0}\frac{\log N_r(A)}{-\log r},
\end{equation*}
where $N_r(A)$ is the maximal number of disjoint $r$-balls with centres in $A$. For a measure $\mu$, we recall that the \emph{lower pointwise dimension of $\mu$ at $x$} is
\begin{equation*}
  \ldimloc(\mu,x) = \liminf_{r \downarrow 0} \frac{\log\mu(B(x,r))}{\log r}.
\end{equation*}
We assume that the reader is familiar with these concepts, and refer to \cite{MR1449135} for a basic introduction.

We say that $\Phi \colon [0,1] \to \R^d$ is an \emph{$\alpha$-bi-H\"older map} if there are constants $c_1,c_2>0$ such that
\begin{equation*}
  c_1 \le \frac{\|\Phi(x)-\Phi(y)\|}{|x-y|^\alpha} \le c_2
\end{equation*}
for all $x,y \in [0,1]$ with $x \ne y$. The choice of the norm in the equation above does not play a role, so for convenience we use the maximum norm. The following lemma is a direct consequence of the Assouad embedding theorem \cite{Assouad1983}, which is a celebrated result in the embedding theory of metric spaces---see the book of Heinonen \cite{Heinonen2001} for a proof.

\begin{lemma} \label{lemma:phi-existence}
  For any $0 < \alpha < 1$, there exist $d\in\N$ and an $\alpha$-bi-Hölder map $\Phi \colon [0,1] \to \R^d$.
\end{lemma}

We are interested in the size of level sets of $\alpha$-Hölder functions. These appear as horizontal slices of the graph of the function or equivalently as fibers of the projection of the graph onto the $y$-axis. We will also consider slices and projections in other directions so for $\theta \in [0,2\pi]$ we let $\proj_\theta \colon \R^2 \to \R$ be the orthogonal projection defined by
\begin{equation*}
  \proj_\theta(x,y) = x\cos(\theta)+y\sin(\theta)
\end{equation*}
for all $(x,y) \in \R^2$.

\subsection{Potential theory and Fourier analysis}\label{sec:potential-fourier}
In the proof of Theorem \ref{thm:main}, we will employ two powerful toolsets from geometric measure theory: Fourier analysis and potential theory. Let us briefly recall the required definitions.

Given a measure $\mu$ on $\R^d$ and $t>0$, we define the \emph{$t$-energy of $\mu$} by
\begin{equation*}
  I_t(\mu)=\iint \|x-y\|^{-t}\dd\mu(x)\dd\mu(y).
\end{equation*}
The following is a basic lemma connecting the energies of measures to Hausdorff dimensions of their supports. For the proof, see i.e. \cite[Theorem 4.13]{MR1102677}.
\begin{lemma}\label{lemma:energy-dim}
  Let $F\subset \R^d$. If there is a measure $\mu$ supported on $F$ satisfying $I_s(\mu)<\infty$, then $\dimh(F)\geq s$.
\end{lemma}

The Fourier transform of a finite Borel measure $\mu$ on $\R$ is defined by
\begin{equation*}
  \hat\mu(\xi) = \int_{\R} e^{i\xi x} \dd\mu(x)
\end{equation*}
for all $\xi \in \R$ and the Fourier transform of an integrable function $\psi \colon \R^d \to \R$ by
\begin{equation*}
  \hat\psi(\zeta) = \int_{\R^d} \psi(\ttt) e^{i\langle \ttt,\zeta \rangle} \dd\ttt
\end{equation*}
for all $\zeta \in \R^d$, where $\langle \,\cdot\,,\cdot\, \rangle$ denotes the standard inner product on $\R^d$. We recall the following lemma, which easily follows from the Riemann-Lebesgue lemma and integration by parts. Here and hereafter $C^{\infty}_0(\R^d)$ denotes the space of smooth functions with compact support on $\R^d$.

\begin{lemma}\label{lemma:riemann-lebesgue}
  For any $\psi\in C^{\infty}_0(\R^d)$ and $s>0$, there is a constant $C=C(\psi,s)>0$, such that
  \begin{equation*}
    |\hat\psi(\zeta)|\leq C(1+\|\zeta\|)^{-s}
  \end{equation*}
  for all $\zeta\in\R^d$.
\end{lemma}

There are intimate connections between the regularity of a measure and the decay rate of its Fourier transform. The book by Mattila \cite{Mattila2015} provides a thorough introduction to the use of Fourier analytic methods in geometric measure theory and fractal geometry. The following lemma connecting the average decay rate of the Fourier transform to the regularity of the measure will be integral in the proof of \cref{thm:main}\eqref{it:main1}. The proof of the lemma can be found for example in \cite[Theorem 5.4]{Mattila2015}.

\begin{lemma}\label{lemma:sobolev-dim}
  Let $\mu$ be a finite Borel measure with compact support on $\R$. If
  \begin{equation*}
    \int_{\R} |\xi|^\beta |\hat\mu(\xi)|^2 \dd\xi < \infty
  \end{equation*}
  for some $\beta>1$, then $\mu$ is absolutely continuous with bounded continuous density.
\end{lemma}

\subsection{Measurability}\label{sec:technical}
For completeness, let us provide the technical details on the measurability of the sets in \cref{thm:main}. Using standard techniques from i.e.\ \cite{Kechris1995} one is able to show that the set in \cref{thm:main}\cref{it:main2} is a Borel set. However, it turns out that the set in \cref{thm:main}\cref{it:main1} is merely co-analytic, which is why we need the following technical result from \cite{Solecki1996}.

\begin{theorem}[Solecki]\label{thm:solecki}
  Let $G$ be an abelian Polish group and let $A\subset G$ be a co-analytic set. If there exists a Borel probability measure $\mu$ on $G$ such that for any $x\in X$ we have $x+s\in A$ for $\mu$-almost all $s\in\spt(\mu)$, then $A$ is prevalent.
\end{theorem}

This theorem together with the following lemma justifies that it is enough to construct probe spaces for the sets in \cref{thm:main}.

\begin{lemma}\label{lemma:technical}
  For any $0<\alpha<1$, the set
  \begin{equation*}
    \mathcal{F}_1 = \{f\in\CC^{\alpha}([0,1])\colon \udimm(f^{-1}(\{y\}))\leq 1-\alpha\text{ for all }y\in\R\}
  \end{equation*}
  is co-analytic and the set
  \begin{equation*}
    \mathcal{F}_2 = \{f\in\CC^{\alpha}([0,1])\colon \LL^1(\{y\in f([0,1])\colon \dimh(f^{-1}(\{y\}))=1-\alpha\})>0\},
  \end{equation*}
  is Borel.
\end{lemma}
\begin{proof}
  Let $\mathcal{K}([0,1])$ denote the space of compact subsets of $[0,1]$ equipped with the Hausdorff distance and consider the function $\Delta_{\mathrm{M}}\colon \mathcal{K}([0,1])\to\R$ defined by
  \begin{equation*}
    \Delta_{\mathrm{M}}(K)=\udimm(K).
  \end{equation*}
  It was shown in \cite{MattilaMauldin1997} that $\Delta_{\mathrm{M}}$ is a Borel function. Consider also the function $S\colon \CC^\alpha([0,1])\times \R\to \mathcal{K}([0,1])$ defined by
  \begin{equation*}
    S(f,y)=f^{-1}(\{y\}),
  \end{equation*}
  which is easily seen to be continuous. Therefore, the set
  \begin{equation*}
    \Delta = \{(f,y)\in\CC^\alpha([0,1])\times \R\colon \udimm(f^{-1}(\{y\}))> 1-\alpha\}=(\Delta_{\mathrm{M}}\circ S)^{-1}((1-\alpha,\infty))
  \end{equation*}
  is Borel. Recalling that the projection $\Pi\colon \CC^\alpha([0,1])\times \R\to \CC^\alpha([0,1])$ defined by $\Pi(f,y)=f$ is continuous, we see that the set
  \begin{equation*}
    \mathcal{F}_1=\CC^\alpha([0,1])\setminus\Pi(\Delta),
  \end{equation*}
  is co-analytic.

  For $\mathcal{F}_2$, consider the function $\Delta_{\mathrm{H}}\colon \mathcal{K}([0,1])\to\R$ defined by
  \begin{equation*}
    \Delta_{\mathrm{H}}(K)=\dimh(K),
  \end{equation*}
  which was also shown to be Borel in \cite{MattilaMauldin1997}. Define for each $f\in\CC^\alpha([0,1])$ the function $S_f\colon \R\to\mathcal{K}([0,1])$ by $S_f(y)=S(f,y)$, and note that the set
  \begin{equation*}
    B_f = \{y\in\R\colon \dimh(f^{-1}(\{y\}))=1-\alpha\}=(\Delta_{\mathrm{H}}\circ S_f)^{-1}(\{1-\alpha\})
  \end{equation*}
  is Borel for all $f\in\CC^\alpha([0,1])$. Moreover, it follows from \cite[Exercise 17.36]{Kechris1995} that the map $L\colon \CC^\alpha([0,1])\to \R$ defined by
  \begin{equation*}
    L(f)=\LL(B_f)
  \end{equation*}
  is a Borel function and therefore, the set $\mathcal{F}_2=L^{-1}((0,\infty])$ is Borel.
\end{proof}

\section{Proof of the main result} \label{sec:proof}
In this section we prove our main result, \cref{thm:main}. To simplify the exposition, we fix some notation for the remainder of this section. Let us fix $0<\alpha<1$ and an $\alpha$-bi-Hölder map $\Phi\colon [0,1]\to\R^d$ given by \cref{lemma:phi-existence}. For a given $f\in C^{\alpha}([0,1])$ and $\ttt\in\R^d$ we define the function $f_{\ttt}\colon [0,1]\to\R$ by
\begin{equation*}
  f_{\ttt}(x)=f(x)+\langle\ttt, \Phi(x)\rangle.
\end{equation*}
It is evident by the triangle inequality that $f_{\ttt}\in C^{\alpha}([0,1])$. Moreover, the space
\begin{equation*}
  \SS = \{x\mapsto \langle\ttt, \Phi(x)\rangle\colon \ttt\in \R^d\}
\end{equation*}
is a finite dimensional subspace of $\CC^{\alpha}([0,1])$. The space $\SS$ will play the role of the probe space in our results: to show that a property holds for prevalent functions, we show that for any $f\in\CC^{\alpha}([0,1])$, the property holds for the function $f_{\ttt}$ for $\LL^d$-almost every $\ttt\in\R^d$.

The proofs in this section rely on the potential theoretic and Fourier analytic methods which we recalled in \cref{sec:potential-fourier}. We will apply these tools to the natural measures $\mu_\ttt$ supported on the graphs of the functions $f_{\ttt}$, which we define by
\begin{equation*}
  \mu_{\ttt}=(\id,f_{\textbf{t}})_*\LL^1.
\end{equation*}
In other words, $\mu_{\ttt}$ is the lift of the Lebesgue measure from the unit interval onto the graph of $f_{\ttt}$. Note that the measure $\mu_{0}$ is just the lift of the Lebesgue measure onto the graph of $f$, and for convenience, we let $\mu=\mu_{0}$. The measures $\mu_{\ttt}$ of course depend on $f$ and $\Phi$ in addition to $\ttt$, but in our applications $f$ and $\Phi$ are fixed and should be clear from the context whereas $\ttt$ varies, which justifies our notation. Finally, for $\theta\in[0,2\pi]$ and a measure $\nu$ on $\R^2$, we let
\begin{equation*}
  \nu^\theta=(\proj_\theta)_*\nu.
\end{equation*}
Note that the projection of $\mu$ onto the $y$-axis is $\mu^{\frac{\pi}{2}}=(\proj_{\frac{\pi}{2}})_*\mu=f_*\LL^1$.

\subsection{Proof of Theorem \ref*{thm:main}(\ref*{it:main1})}
We will start the proof of the first claim of \cref{thm:main} by observing a connection between the Minkowski dimension of the level sets of $f$ and the pointwise dimensions of the projected measure $\mu^{\frac{\pi}{2}}$.
\begin{lemma} \label{thm:dim-lemma}
  Suppose that $0 < \alpha < 1$ and $f \in C^\alpha([0,1])$. Then we have
  \begin{equation*}
    \udimm(f^{-1}(\{y\})) \le 1 - \alpha \ldimloc(\mu^{\frac{\pi}{2}},y)
  \end{equation*}
  for all $y \in f([0,1])$.
\end{lemma}

\begin{proof}
  Let $B(x,r)$ be the open ball centered at $x$ with radius $r$. Let $\{B(x_i,r)\}_{i=1}^{N_r}$ be a maximal $r$-packing of $f^{-1}(\{y\})$. Note that $y = f(x_i)$ and $B(x_i,r) \cap B(x_j,r) = \emptyset$ whenever $i \ne j$. Since $|f(x_i)-f(z)| \le c|x_i-z|^\alpha \le cr^\alpha$ for all $z \in B(x_i,r)$, we have $\mu^{\frac{\pi}{2}}(B(y,cr^\alpha)) = \LL^1(f^{-1}(B(y,cr^\alpha))) \ge \sum_{i=1}^{N_r}\LL^1(B(x_i,r))= 2N_rr$ and therefore
  \begin{align*}
    \frac{\log\mu^{\frac{\pi}{2}}(B(y,cr^\alpha))}{\log r} \le \frac{\log 2N_r}{\log r} + 1.
  \end{align*}
  Taking the limit inferior yields $\alpha \ldimloc(\mu^{\frac{\pi}{2}},y) \le -\udimm(f^{-1}(\{y\})) + 1$ as claimed.
\end{proof}

The following proposition together with \cref{lemma:technical} and \cref{thm:solecki} imply \cref{thm:main}\cref{it:main1}.

\begin{proposition} \label{thm:main-tech}
  Suppose that $0 < \alpha < \frac12$ and $f \in C^\alpha([0,1])$. Then for $\LL^d$-almost all $\ttt \in \R^d$ the function $f_{\ttt} \colon [0,1] \to \R$ satisfies
  \begin{equation*}
    \udimm(f_{\ttt}^{-1}(\{y\})) \le 1-\alpha
  \end{equation*}
  for all $y \in \R$.
\end{proposition}

\begin{proof}
  By \cref{thm:dim-lemma}, it is enough to show that
  \begin{equation}\label{eq:pointwise-dim-lb}
    \ldimloc(\mu_{\ttt}^{\frac{\pi}{2}},y) \ge 1
  \end{equation}
  for all $y \in f_{\ttt}([0,1])$. Since $0 < \alpha < \frac12$, we can choose $\beta>1$ and $s > 2$ such that $2 < \beta+1 < s$ and $\alpha s < 1$. Let $\roo>0$ and let $\psi\in C^\infty_0(\R^d)$ be a smooth function satisfying $0\leq \psi\leq 1$ and $\psi(\ttt)=1$ for all $\ttt\in B(\mathbf{0},\roo)$. By using Fubini's theorem and \cref{lemma:riemann-lebesgue}, we see that there is a constant $C_1=C_1(\psi,s)>0$ such that
  \begin{align*}
    \biggl| \int\limits_{B(\mathbf{0},\roo)} &\int |\xi|^\beta |\hat\mu_{\ttt}^{\frac{\pi}{2}}(\xi)|^2 \dd\xi\dd\textbf{t} \biggr| \\ 
    &\leq \biggl| \iiint\int_0^1\int_0^1 \psi(\textbf{t}) |\xi|^\beta e^{i\xi(f(x)-f(y)+\langle \textbf{t},\Phi(x)-\Phi(y) \rangle)} \dd x\dd y\dd\xi\dd\textbf{t} \biggr| \\ 
    &= \biggl| \iiint\int_0^1\int_0^1 |\xi|^\beta e^{i\xi(f(x)-f(y))} \int \psi(\textbf{t}) e^{i\langle \textbf{t},\xi(\Phi(x)-\Phi(y)) \rangle} \dd\textbf{t}\dd x\dd y\dd\xi \biggr| \\ 
    &= \biggl| \iiint\int_0^1\int_0^1 |\xi|^\beta e^{i\xi(f(x)-f(y))} \hat\psi(\xi(\Phi(x)-\Phi(y))) \dd x\dd y\dd\xi \biggr| \\ 
    &\le \iiint\int_0^1\int_0^1 |\xi|^\beta |\hat\psi(\xi(\Phi(x)-\Phi(y)))| \dd x\dd y\dd\xi \\ 
    &\le \iiint\int_0^1\int_0^1 \frac{C_1|\xi|^\beta}{(1+|\xi|\|\Phi(x)-\Phi(y)\|)^s} \dd x\dd y\dd\xi.
  \end{align*}
  Since we may assume that the constant $c_1$ in the bi-Hölder condition is at most one, by using the fact that $\Phi$ is $\alpha$-bi-Hölder and that $|x-y|\leq1$ we have
  \begin{align*}
    \int\int_0^1\int_0^1& \frac{C_1|\xi|^\beta}{(1+|\xi|\|\Phi(x)-\Phi(y)\|)^s} \dd x\dd y\dd\xi\leq\int\int_0^1\int_0^1 \frac{C_1|\xi|^\beta}{(1+c_1|\xi||x-y|^{\alpha})^s} \dd x\dd y\dd\xi\\
    &\leq \int\int_0^1\int_0^1 \frac{C_1|\xi|^\beta}{\left(\frac{1}{|x-y|^{\alpha}}+c_1|\xi|\right)^s|x-y|^{\alpha s}} \dd x\dd y\dd\xi\\
    &\leq\int\frac{C_1|\xi|^\beta}{c_1(1+|\xi|)^s}\dd\xi\int_0^1\int_0^1\frac{1}{|x-y|^{\alpha s}} \dd x\dd y<\infty
  \end{align*}
  by the choice of $\beta$ and $s$. It follows from \cref{lemma:sobolev-dim} that for $\LL^d$-almost every $\ttt$, the measure $\mu_{\ttt}^{\frac{\pi}{2}}$ is absolutely continuous with bounded continuous density and therefore \cref{eq:pointwise-dim-lb} holds for all $y\in f_{\ttt}([0,1])$.
\end{proof}

\subsection{Proof of Theorem \ref*{thm:main}(\ref*{it:main2})}
We shall now proceed to the proof of the second claim. Our strategy is to concentrate measures onto the slices of the graph using Rohlin's disintegration theorem and to use potential theory to bound the dimension of the slice measures from below. The purpose of the following proposition is to establish absolute continuity for the measures $\mu_{\ttt}^\theta$, for almost all parameters, which will be crucial in the proof of \cref{thm:main}\eqref{it:main2}. The proposition also provides a short proof of \cref{thm:prevalent-dim}.

\begin{proposition} \label{prop:graph-big}
  For each $\ttt_0 \in \R^d$ and $\roo>0$ the function $f_{\ttt} \colon [0,1] \to \R$ satisfies
  \begin{equation*}
    \int\limits_{B(\ttt_0,\roo)} I_{2-\beta}(\mu_{\ttt}) \dd\ttt < \infty
  \end{equation*}
  for all $\beta>\alpha$. In particular, $\dimh(\graph(f_{\ttt})) = 2-\alpha$ for $\LL^d$-almost all $\ttt \in \R^d$ and 
  \begin{equation*}
    \mu_{\ttt}^{\theta} \ll \LL^1
  \end{equation*}
  for $\LL^d$-almost all $\ttt \in \R^d$ and $\LL^1$-almost all $\theta\in[0,2\pi]$.
\end{proposition}

\begin{proof}
  By the definition of $\mu_{\ttt}$ and Fubini's theorem,
  \begin{align*}
    \int\limits_{B(\ttt_0,\roo)} I_{2-\beta}(\mu_{\ttt}) \dd\ttt&= \int\limits_{B(\ttt_0,\roo)} \int_0^1\int_0^1 \frac{1}{\|(x,f_{\textbf{t}}(x))-(y,f_{\textbf{t}}(y))\|^{2-\beta}} \dd x\dd y\dd\ttt \\ 
    &= \int\limits_{B(\ttt_0,\roo)} \int_0^1\int_0^1 \frac{\dd x\dd y\dd\ttt}{\bigl(|x-y|^2 + (f(x)-f(y)+\langle \ttt,\Phi(x)-\Phi(y) \rangle)^2\bigr)^{1-\frac{\beta}{2}}} \\
    &= \int_0^1\int_0^1 \int\limits_{B(\ttt_0,\roo)} \frac{\dd\ttt\dd x\dd y}{|x-y|^{2-\beta}\Bigl(1 + \Bigl(\frac{f(x)-f(y)+\langle \ttt,\Phi(x)-\Phi(y) \rangle}{|x-y|}\Bigr)^2\Bigr)^{1-\frac{\beta}{2}}}.
  \end{align*}
  For a given pair $x,y$ we let $t_0(x,y)$ be the $i$th coordinate of $\ttt_0$, where $i$ is the index satisfying $|\Phi_i(x)-\Phi_i(y)|=\|\Phi(x)-\Phi(y)\|$. By symmetry, we may assume that $|\Phi_i(x)-\Phi_i(y)|=\Phi_i(x)-\Phi_i(y)$ and therefore by noting that $\langle \ttt,\Phi(x)-\Phi(y) \rangle\geq t_i\|\Phi(x)-\Phi(y)\|$ for any $\ttt\in\R^d$, we have
  \begin{equation*}
    \int\limits_{B(\ttt_0,\roo)} I_{2-\beta}(\mu_{\ttt}) \dd\ttt \leq \int_0^1\int_0^1 \int\limits_{t_0(x,y)-\roo}^{t_0(x,y)+\roo} \frac{(2\roo)^{d-1}\dd t\dd x\dd y}{|x-y|^{2-\beta}\Bigl(1 + \Bigl(\frac{f(x)-f(y)}{|x-y|}+\frac{t\|\Phi(x)-\Phi(y)\|}{|x-y|}\Bigr)^2\Bigr)^{1-\frac{\beta}{2}}}.
  \end{equation*}
  By the change of variable
  \begin{equation*}
    w = \frac{f(x)-f(y)}{|x-y|}+\frac{t\|\Phi(x)-\Phi(y)\|}{|x-y|}
  \end{equation*}
  and recalling that $\Phi$ is $\alpha$-bi-H\"older, we find constants $C_2,C_3>0$ such that
  \begin{align*}
    \int\limits_{B(\ttt_0,\roo)} I_{2-\beta}(\mu_{\ttt}) \dd\ttt &\le \int\int_0^1\int_0^1 \frac{C_2(2\roo)^{d-1}\dd w}{|x-y|^{2-\beta}(1+w^2)^{1-\frac{\beta}{2}}} \cdot \frac{|x-y|}{\|\Phi(x)-\Phi(y)\|} \dd x\dd y \\ 
    &\le \int \frac{C_3(2\roo)^{d-1}\dd w}{(1+w^2)^{1-\frac{\beta}{2}}} \int_0^1\int_0^1 \frac{1}{|x-y|^{1-(\beta-\alpha)}} \dd x\dd y < \infty,
  \end{align*}
  where the last integral is finite, since $1-(\beta-\alpha)<1$.

  The upper bound on the dimension of $\graph(f_{\ttt})$ is standard, see e.g. \cite[Corollary 11.2]{MR1102677}. For the lower bound, by Lemma \ref{lemma:energy-dim}, the first part of the proof implies that for any $\beta>\alpha$, we have $\dimh(\graph(f_{\ttt})) \geq 2-\beta$ for $\LL^d$-almost all $\ttt$. By taking $\beta\to\alpha$ along a countable sequence, we obtain the lower bound. Finally, the first part combined with the version of Marstrand's projection theorem for measures, see for example \cite[Theorem 9.7]{Mattila1995}, shows that $\mu_{\ttt}^\theta\ll\LL^1$ for $\LL^1$-almost all $\theta\in[0,2\pi]$.
\end{proof}

Recall that our aim is to show that $\dimh(f^{-1}(\{y\}))=1-\alpha$, for positively many $y$. The following proposition contains the main technical part of the argument.

\begin{proposition} \label{thm:graph-slices-big}
  The function $f_{\ttt} \colon [0,1] \to \R$ satisfies
  \begin{equation*}
    \dimh(\graph(f_{\ttt}) \cap \proj_\theta^{-1}(\{y\})) = 1-\alpha
  \end{equation*}
  for $\LL^d$-almost all $\ttt \in \R^d$, $\LL^1$-almost all $\theta \in [0,2\pi]$, and $\mu_{\ttt}^\theta$-almost all $y \in \proj_\theta(\graph(f_{\ttt}))$.
\end{proposition}

\begin{proof}
  By Marstrand's slicing theorem \cite[Theorem 3.3.1]{BishopPeres2016}, it suffices to show that
  \begin{equation} \label{eq:graph-slices-big-claim8}
    \dimh(\graph(f_{\ttt}) \cap \proj_\theta^{-1}(\{y\})) \ge 1-\alpha
  \end{equation}
  for $\LL^d$-almost all $\ttt \in \R^d$, $\LL^1$-almost all $\theta \in \R$, and $\mu_{\ttt}^\theta$-almost all $y \in \proj_\theta(\graph(f_{\ttt}))$. By \cref{prop:graph-big},
  \begin{equation*} 
    \mu_{\ttt}^\theta \ll \LL^1
  \end{equation*}
  for $\LL^d$-almost all $\ttt \in \R^d$ and $\LL^1$-almost all $\theta \in [0,2\pi]$, and therefore the Radon-Nikodym theorem \cite[Theorem 2.12(2)]{Mattila1995} shows that there exists an integrable function $h_{\ttt}^\theta = \mathrm{d}\mu^{\theta}_{\ttt}$, satisfying $\mu_{\ttt}^\theta(A) = \int_A h_{\ttt}^\theta \dd\LL^1$. We apply Rohlin's disintegration theorem \cite[Theorem 2.1]{Simmons2012} to the projection $\proj_{\theta}\colon \R^2\to\R$ and the measure $\mu_{\ttt}$, and let $\{(\mu_{\textbf{t}})_{\theta,y}\}$ denote the obtained family of conditional measures, where each $(\mu_{\textbf{t}})_{\theta,y}$ is supported on $\graph(f_{\ttt})\cap\proj_{\theta}^{-1}(\{y\})$. By \cite{Simmons2012}, we have
  \begin{equation*}
    \textrm{d}\mu_{\ttt} = \mathrm{d}(\mu_{\ttt})_{\theta,y} \dd\mu^{\theta}_{\ttt}.
  \end{equation*}
  Write $\mathrm{d}(\nu_{\ttt})_{\theta,y} = h_{\ttt}^\theta(y)\dd(\mu_{\ttt})_{\theta,y}$, and notice that
  \begin{equation} \label{eq:graph-slices-big-claim5}
    \mathrm{d}\mu_{\ttt} = \mathrm{d}(\nu_{\textbf{t}})_{\theta,y} \dd y.
  \end{equation}
  Furthermore, by \cite[Theorem 2.2]{Simmons2012}, we have
  \begin{equation} \label{eq:graph-slices-big-claim6}
    \int g(a) \dd(\mu_{\ttt})_{\theta,y}(a) = \lim_{r \downarrow 0} \frac{\int_{B_\theta^\top(y,r)} g(a) \dd\mu_{\ttt}(a)}{\mu^{\theta}_{\ttt}(B(y,r))} = \lim_{r \downarrow 0} \frac{1}{2r} \frac{\int_{B_\theta^\top(y,r)} g(a) \dd\mu_{\ttt}(a)}{h_{\ttt}^\theta(y)}
  \end{equation}
  for all integrable functions $g$, where
  \begin{equation*}
    B_\theta^\top(y,r) = \proj_\theta^{-1}(B(\proj_\theta(y),r)).
  \end{equation*}
  Let now $\beta>\alpha$. By using Fatou's lemma, \cref{eq:graph-slices-big-claim6}, \cref{eq:graph-slices-big-claim5}, and Fubini's theorem, we get
  \begin{align*}
    \int\limits_{B(\ttt_0,\roo)}& \int\limits_{[0,2\pi]} \int I_{1-\beta}((\nu_{\ttt})_{\theta,y}) \dd y\dd\theta\dd\textbf{t} \\
    &= \int\limits_{B(\ttt_0,\roo)} \int\limits_{[0,2\pi]} \iiint \frac{1}{\|a-b\|^{1-\beta}} \dd(\nu_{\textbf{t}})_{\theta,y}(a) \dd(\nu_{\textbf{t}})_{\theta,y}(b) \dd y\dd\theta\dd\textbf{t} \\
    &= \int\limits_{B(\ttt_0,\roo)} \int\limits_{[0,2\pi]} \iiint \frac{h_{\ttt}^\theta(y)}{\|a-b\|^{1-\beta}} \dd(\mu_{\textbf{t}})_{\theta,y}(a) \dd(\nu_{\textbf{t}})_{\theta,y}(b) \dd y\dd\theta\dd\textbf{t} \\ 
    &\le \liminf_{r \downarrow 0} \frac{1}{2r} \int\limits_{B(\ttt_0,\roo)} \int\limits_{[0,2\pi]} \iint \int\limits_{\;B_\theta^\top(y,r)} \frac{\mathrm{d}\mu_{\ttt}(a)}{\|a-b\|^{1-\beta}} \dd(\nu_{\textbf{t}})_{\theta,y}(b) \dd y\dd\theta\dd\ttt \\ 
    &= \liminf_{r \downarrow 0} \frac{1}{2r} \int\limits_{B(\ttt_0,\roo)} \int\limits_{[0,2\pi]} \int \int\limits_{\;B_\theta^\top(\proj_\theta(b),r)} \frac{1}{\|a-b\|^{1-\beta}} \dd\mu_{\ttt}(a)\dd\mu_{\textbf{t}}(b)\dd\theta\dd\ttt\\
    &=\liminf_{r \downarrow 0} \frac{1}{2r} \int\limits_{B(\ttt_0,\roo)} \iint \frac{\LL^1(\{\theta \in [0,2\pi] \colon |\proj_\theta(a)-\proj_\theta(b)|<r\})}{\|a-b\|^{1-\beta}} \dd \mu_{\ttt}(a)\dd \mu_{\ttt}(b)\dd\ttt.
  \end{align*}
  Recalling that there exists a constant $C>0$ such that for all $a,b\in\R^2$ we have
  \begin{equation*}
    \LL^1(\{\theta \in [0,2\pi] \colon |\proj_\theta(a)-\proj_\theta(b)|<r\})\leq 2Cr\| a-b\|^{-1},
  \end{equation*}
  see i.e. \cite[Lemma 3.11]{Mattila1995}, we obtain
  \begin{align*}
    \liminf_{r \downarrow 0} \frac{1}{2r}& \int\limits_{B(\ttt_0,\roo)} \iint \frac{\LL^1(\{\theta \in [0,2\pi] \colon |\proj_\theta(a)-\proj_\theta(b)|<r\})}{\|a-b\|^{1-\beta}} \dd \mu_{\ttt}(a)\dd \mu_{\ttt}(b)\dd\ttt\\
    &\le \int\limits_{B(\ttt_0,\roo)} \iint \frac{C}{\|(x,f_{\textbf{t}}(x))-(z,f_{\textbf{t}}(z))\|^{2-\beta}} \dd x\dd z\dd\ttt<\infty,
  \end{align*}
  where the last integral is finite by \cref{prop:graph-big}. Therefore, $I_{1-\beta}((\nu_{\ttt})_{\theta,y})<\infty$ for $\LL^d$-almost all $\ttt \in \R^d$, $\LL^1$-almost every $\theta\in[0,2\pi]$, and $\mu_{\ttt}^\theta$-almost all $y\in\proj_{\theta}(\graph(f_{\ttt}))$. The proof of \cref{eq:graph-slices-big-claim8} concludes as the proof of \cref{prop:graph-big} by noting that $(\nu_{\ttt})_{\theta,y}$ is supported on $\graph(f_{\ttt}) \cap \proj_\theta^{-1}(\{y\})$ and taking $\beta\to\alpha$ along a countable sequence.
\end{proof}

Note that in the previous proposition there is a priori no reason that the result holds for the slice with $\theta=\frac{\pi}{2}$ corresponding to the level sets. However, by adding a dimension to our probe space, we are able to finish the proof of our main result.
\begin{proof}[Proof of \cref{thm:main}\cref{it:main2}]
  For $\theta \ne k\pi$, $k\in\N$, write $f_{\ttt}^\theta(x) = f(x) + \langle \ttt,\Phi(x) \rangle + x\cot(\theta)$, and observe that
  \begin{equation*} 
    f_{\textbf{t}}^\theta(x) = \frac{1}{\sin(\theta)} \proj_\theta \circ (\id,f_{\ttt})(x)
  \end{equation*}
  for all $x \in [0,1]$. Since
  \begin{equation*}
    (f_{\ttt}^\theta)^{-1}(\{y\}) = (\id,f_{\ttt})^{-1}(\graph(f_{\ttt}) \cap \proj_\theta^{-1}(\{y\sin(\theta)\}))
  \end{equation*}
  for all $y \in \R$, \cref{thm:graph-slices-big} shows that
  \begin{equation*}
    \dimh(\graph(f_{\ttt}) \cap \proj_\theta^{-1}(\{y\sin(\theta)\})) = 1-\alpha
  \end{equation*}
  for $\LL^d$-almost all $\ttt \in \R^d$, $\LL^1$-almost all $\theta \in \R$, and $\mu_{\ttt}^\theta$-almost all $y \in f_{\ttt}^\theta([0,1])$. Recalling \cref{lemma:technical} and that $\mu_{\ttt}^\theta \ll \LL^1$ the proof is finished.
\end{proof}

\section{Discussion and an open question}\label{sec:discussion}
To conclude the paper, we discuss the modification of our strategy for Weierstrass-type functions and where our methods fall short in proving \cref{conj:main}. The key thing to note is that if one can construct an $\alpha$-bi-Hölder embedding $\Phi\colon \R^d\to \R$, whose coordinate functions are in $\WW^{\alpha,b}$, then our results hold for prevalent functions in $\WW^{\alpha,b}$. However, finding such a construction seems to be quite difficult and our efforts for doing this for Weierstrass-type functions have been unsuccessful. Therefore, we pose the construction as an open problem.

Recall the definition of the space of Weierstass-type functions $\WW^{\alpha,b}$ from \cref{sec:intro}. Let $\mathcal{G}=\{g_1,\ldots,g_m\}$ be a finite collection of Lipschitz functions on $\mathbb{S}^1$. The \emph{Weierstrass embedding} $\Phi_{\mathcal{G}}^{\alpha,b}\colon \mathbb{S}^1\to\R^m$ associated to $\mathcal{G}$, $b\in\N$, and $0<\alpha<1$ is defined by
\begin{equation*}
  \Phi_{\mathcal{G}}^{\alpha,b}(x)=(W_{g_1}^{\alpha,b}(x),W_{g_2}^{\alpha,b}(x),\ldots,W_{g_m}^{\alpha,b}(x)).
\end{equation*}
We note that if $\Phi=\Phi_{\mathcal{G}}^{\alpha,b}$ is a Weierstrass embedding, then for any $\ttt\in\R^m$ and $W\in\WW^{\alpha,b}$, the function $W_{\Phi,\ttt}\colon [0,1]\to \R$ defined by
\begin{equation*}
  W_{\Phi,\ttt}(x)=W(x)+\langle\ttt,\Phi(x)\rangle
\end{equation*}
is in $\WW^{\alpha,b}$. An affirmative answer to the following question implies \cref{conj:main}.

\begin{question}\label{ques:bi-holder}
Given $b\in\N$ and $0<\alpha<1$, does there exist a finite collection $\mathcal{G}=\{\varphi_1,\ldots,\varphi_m\}$ Lipschitz functions on $\mathbb{S}^1$ such that the Weierstrass embedding $\Phi_{\mathcal{G}}^{\alpha,b}\colon \mathbb{S}^1\to\R^m$ is $\alpha$-bi-Hölder?
\end{question}

Some evidence on the answer to the question being affirmative was given in \cite{HuLau1993}, where the authors show that many Weierstrass-type functions have many points where the function satisfies a reverse-Hölder condition. To be more precise, for $W=W_{g}^{\alpha,b}\in\WW^{\alpha,b}$, $\varepsilon>0$, $x\in [0,1]$, and any interval $I_x\subset [0,1]$ containing $x$, let
\begin{equation*}
  E(W,\varepsilon,I_x)=\{y\in I_x\colon | W(x)-W(y)|>\varepsilon | x-y|^{\alpha}\}.
\end{equation*}
It was shown in \cite[Theorem 4.1]{HuLau1993} that if the function $V\colon \mathcal{S}^1\to\R$ defined by $V(x)=\sum_{k=-\infty}^{\infty}b^{-\alpha k}g(b^kx)$ is not identically zero, then there are constants $c>0$ and $\varepsilon>0$, such that for any $x\in [0,1]$ and any small enough interval $I_x$ containing $x$, we have
\begin{equation*}
  \LL^1(E(W,\varepsilon,I_x))\geq c\LL^1(I_x).
\end{equation*}
This leads us to believe that if one chooses the functions $g_i$ ``independently'' enough, there should be a way to construct a bi-Hölder Weierstrass embedding but despite our efforts, we have not been able to find such a construction.



\end{document}